\documentclass{amsart}

\newtheorem{theorem}{Theorem}[section]
\newtheorem{lemma}[theorem]{Lemma}

\theoremstyle{definition}
\newtheorem{definition}[theorem]{Definition}

\newtheorem{proposition}[theorem]{Proposition}
\theoremstyle{remark}
\newtheorem{remark}[theorem]{Remark}

\numberwithin{equation}{section}

\begin{document}
	
	\title{ Fractal dimension for a class of complex-valued fractal interpolation functions}
	

	
	\author{Manuj Verma}
	\address{Department of Mathematics, IIT Delhi, New Delhi, India 110016}
	\email{manujverma123@gmail.com}
	\author{Amit Priyadarshi}
	\address{Department of Mathematics, IIT Delhi, New Delhi, India 110016}
	
	\email{priyadarshi@maths.iitd.ac.in}
	\author{Saurabh Verma}
	\address{Department of Applied Sciences, IIIT Allahabad, Prayagraj, India 211015 }
	\email{saurabhverma@iiita.ac.in}
	
	

	\begin{abstract}
		There are many research papers dealing with fractal dimension of real-valued fractal functions in the recent literature. The main focus of the present paper is to study fractal dimension of complex-valued functions. This paper also highlights the difference between dimensional results of the complex-valued and real-valued fractal functions.
		In this paper, we study the fractal dimension of the graph of complex-valued function $g(x)+i h(x)$, compare its fractal dimension with the  graphs of functions $g(x)+h(x)$ and $(g(x),h(x))$ and also obtain some bounds. Moreover, we study the fractal dimension of the graph of complex-valued fractal interpolation function associated with a germ function $f$, base function $b$ and scaling functions $\alpha_k$.
		
	\end{abstract}
	
	\maketitle

	
	
	\section{INTRODUCTION}

	Fractal dimension is one of the major themes in Fractal Geometry. Estimation of fractal dimension of sets and graphs has received a lot of attention in the literature \cite{Fal}.
	In 1986, Barnsley \cite{MF1} introduced the idea of fractal interpolation functions(FIF) and computed the Hausdorff dimension of affine FIF. Estimation of box dimension for a class of affine FIFs presented in \cite{MF6,B&H,M&H}. Several authors \cite{MF3,HM,JV1,Kono,VV} also calculated the fractal dimension of the graph of FIF. In 1991, Massopust \cite{MP} estimated the box dimension of graph of vector-valued FIF. Later  Hardin and Massopust \cite{HM} constructed fractal interpolation functions from $\mathbb{R}^n$ to $\mathbb{R}^m$ and also given the formula for calculating the box dimension.  We encourage the reader to see some recent works on fractal dimension of fractal functions defined on different domains such as Sierpinski gasket \cite{SP,TV}, rectangular domain \cite{SS2} and interval \cite{JV1,VV}. To the best of our knowledge, we may say that there is no work available for dimension of complex-valued fractal functions. Here we give some basic results for complex-valued FIF and provide some results to convince the reader that there is some difference between dimensional result of the complex-valued and real-valued fractal functions.  
	\par In 1986, Mauldin and Williams \cite{w&l} were the poineers who studied the problem of decomposition of the continuous functions in terms of fractal dimensions. They proved the existence of decomposition of  any continuous function on $[0,1]$ into sum of two continuous functions, where each have Hausdorff dimension one. Later in 2000, Wingren \cite{pt} gave a technique  to construct the above decomposition of Mauldin and Williams. Moreover, he proved the same type result of Mauldin and Williams for the lower box dimension. Bayart and Heurteaux \cite{BFY} also  proved similar result for Hausdorff dimension $\beta=2$, and raised the question for $\beta\in [1,2]$. Recently in 2013, Jia Liu and Jun Wu \cite{LW} solved the question which was raised by Bayart and Haurteaux. More preciesly
	they proved that, for any $\beta \in [1,2]$, each continuous function on $[0,1]$ can be decomposed into sum of two continuous functions, where each have Hausdorff dimension $\beta.$ Falconer and Fraser \cite{F&F} found an upper bound for upper box dimension of graph of sum of two continuous functions which depends on dimension of both of graphs.  
	\par
	In \cite{LW,F&F}, it is clear that the Hausdorff dimension of graph of $g+h$ does not depend on the Hausdorff dimension of graph of $g$ and $h$ whereas, the upper box dimension depends on both.
	Motivated from this we think about the behaviour of Hausdorff dimension of  graph of $g+ih$, whether it  depends on the Hausdorff dimension of graphs of $g$ and $h $ or not? We obtained an affirmative  answer for this question. Also the upper box dimension of  $g+ih$ depends on the upper box dimensions of $g$ and $h$ which is quite different from the upper box dimension of $g+h$. Finally, we studied some relations between fractal dimensions of the graphs of $g(x)+ i h(x)$, $g(x)+ h(x)$ and $(g(x), h(x))$.
	
	\par
	The paper is organized as follows. In upcoming section \ref{se2}, we give some preliminary results and required definition for next section. Section \ref{se3} consists of some dimensional results for complex-valued continuous functions and FIFs. In this section, first we establish some propositions and lemmas to form a relation between the fractal dimension of complex-valued and real-valued continuous functions. After that, we determine bound of Hausdorff dimension of $\alpha$-fractal function under some assumption.
	We also obtain some conditions under which $\alpha$-fractal function becomes a H\"{o}lder continuous function and bounded variation function, and calculate its fractal dimension.
	
	\section{preliminaries}\label{se2}
	\begin{definition}
		Let $F$ be a subset of a metric space $(Y,d)$. The Hausdorff dimension of $F$ is defined as follows
		$$ \dim_H{F}=\inf\{\beta>0: \text{for every}~\epsilon>0,~\text{there is a cover}~~ \{V_i\}~\text{of}~F~\text{ with}\sum |V_i|^\beta<\epsilon \}$$

	\end{definition}
	\begin{definition}
		The box dimension of a non-empty bounded subset $F$ of a metric space $(Y,d)$ is defined as
		$$\dim_{B}F=\lim_{\delta \to 0}\frac{\log{N_{\delta}(F)}}{-\log\delta},$$
		where $N_\delta(F)$ is the minimum number of sets of diameter $\delta>0$ that can cover $F.$ 
		If this limit does not exist, then limsup and liminf is known as upper and lower box dimension respectively.
	\end{definition}
	\begin{definition}
		For $r> 0$ and $t\geq0$, let
		$P^t_r(F)=\sup\big \{\sum^{}_i|B_i|^t  \big\}$, where $\{B_i\}$ is a collection of disjoint balls of radii at most $r$ with centres in $F$.
		As $r$ decreases, $P^t_r$ also decreases. Therefore the limit
		$$P^t_0(F)= \lim_{r \to 0}P^t_r(F)$$
		exists. We define
		$$P^t(F)=\inf\bigg\{\sum_iP^t_0(F_i): F\subset\bigcup^{\infty}_{i=1}F_i\bigg\},$$
		and it is known as the $t$-dimensional packing measure. Packing dimension is defined as follows:
		$$\dim_{P}(F)=\inf\{t\geq 0 : P^{t}(F)=0\}=\sup\{t\geq 0 : P^{t}(F)=\infty\}.$$
	\end{definition}
	\textbf{Note-} We denote graph of function $f$ by $G(f)$ throughout this paper.
	\begin{remark}
		$f=f_1+if_2 :[a,b] \to \mathbb{C} $ is a H\"older continuous function with H\"{o}lder exponent $\sigma$ if and only if $f_i$ is also H\"older continuous function with H\"{o}lder exponent $\sigma $ for every $i=1,2$.
	\end{remark}
	\begin{theorem}\cite{Fal}
		If $f: [0,1] \to \mathbb{R}$ is a H\"{o}lder continuous function with the H\"{o}lder exponent $ \sigma\in (0,1)$. Then   $\overline{\dim}_B(G(f))\leq 2-\sigma. $ 
		
	\end{theorem}
	
	
	
	\subsection{Iterated Function Systems} Let $(Y, d)$ be a complete metric space, and we denote the family of all nonempty compact subsets of $Y$ by $H(Y)$ . For any $A_1,A_2\in H(Y)$, we define the Hausdorff metric by 
	$$h(A_1,A_2) = \inf\{\delta>0 : A_1\subset {A_2}_\delta~~\text{and}~~ A_2 \subset {A_1}_\delta \} ,$$
	where ${A_1}_\delta$ and ${A_2}_\delta$ denote the $\delta $-neighbourhood of sets $A_1$ and $A_2$, respectively. It is well-known that $(H(Y),h)$ is a complete metric space. 
	\\\textbf{Note-} A map $\theta: (Y,d) \to (Y,d) $  is called a contraction if there exists a constant  $c<1$ such that 
	$$d(\theta(a),\theta(b)) \le c~~d(a,b),~~\forall~~~ a , b \in Y.$$
	\begin{definition}
		The system $\mathcal{I}=\big\{(Y,d); \theta_1,\theta_2,\dots,\theta_N \big\}$ is called an iterated function system (IFS), if each $\theta_i$ is a contraction self-map on $Y$ for $i\in \{1,2,\dots,N\}$.
	\end{definition}
	\textbf{Note-} Let $\mathcal{I}=\big\{(Y,d); \theta_1,\theta_2,\dots,\theta_N \big\}$ be an IFS. We define a mapping $S$ from $H(Y)$ into $H(Y)$ given by
	$$ S(A) = \cup_{i=1}^N \theta_i (A).$$
	The map $S$ is a contraction
	map under the Hausdorff metric $h$. If $(Y,d)$ is a complete metric space therefore, by Banach contraction principle, there exists a unique $E\in H(Y)$ such that $ E = \cup_{i=1}^N \theta_i (E) $ and it is called the attractor of the IFS. We refer the reader to see \cite{MF2,Fal}
	for details.
	\begin{definition}
		We say that an IFS $\mathcal{I}=\{(Y,d);\theta_1,\theta_2,\dots,\theta_N\}$ satisfies the open set condition(OSC) if there is a non-empty open set $O$ with $\theta_i(O) \subset O~~\forall~i\in \{1,2,\cdots,N\}$  and  $ \theta_i(O)\cap \theta_j(O)= \emptyset $ for $i\ne j$. Moreover, if $O \cap E \ne \emptyset,$  where $E$ is the attractor of the $\mathcal{I}$, then we say that  $\mathcal{I}$ satisfies the strong open set condition(SOSC). If $\theta_i(E)\cap \theta_j(E)=\emptyset$ for $i\ne j$, then we say that the IFS $\mathcal{I}$ satisfies the  strong separation condintion(SSC).  
	\end{definition}
	\subsection{Fractal Interpolation Functions}
	Consider a set of data points $ \{(x_i,y_i)\in \mathbb{R}\times\mathbb{C} : i=1,2,\dots,N\} $ with $x_1<x_2<\dots <x_N$. Set $ T = \{1,2,...,N-1\}$ and   $J= [x_1, x_N] .$ For each $ k \in T,$ set $J_k= [x_k, x_{k+1}]$ and let $P_k: J \rightarrow J_k $ be a contractive homeomorphism  satisfying
	$$ P_k(x_1)=x_k,~~P_k(x_N)=x_{k+1}.$$
	For each $ k\in T $, let $\Psi_k: J\times \mathbb{C} \rightarrow \mathbb{C} $ be a continuous  map such that 
	$$ |\Psi_k(t,z_1) - \Psi_k(t,z_2)| \leq \tau_k |z_1- z_2| ,$$
	$$ \Psi_k(x_1,y_1)=y_k, \Psi_k(x_N,y_N)=y_{k+1},$$
	where $(t,z_1), (t,z_2) \in J\times \mathbb{C} $ and $ 0 \leq \tau_k < 1.$  In particular, we can take for each $k\in T$,
	$$P_k(t)=a_k t+ d_k, \quad \Psi_k(t,y) = \alpha_k y + q_k (t).$$
	Constants $a_k$ and $d_k$ are uniquely determined by the condition $ P_k(x_1)=x_k, P_k(x_N)=x_{k+1}.$ The multiplier $\alpha_k$ is called the scaling factor, which satisfies $-1< \alpha_k <1$ and $q_k:J\rightarrow \mathbb{C}$ is a continuous function such that $q_k(x_1)=y_k-\alpha_k y_1$ and $q_k(x_N)=y_{k+1}-\alpha_k y_N$.
	Now for each $k \in T $, we define functions $W_k:J\times \mathbb{C} \rightarrow  J\times \mathbb{C} $ by  $$W_k(t,y)=\big(P_k(t),\Psi_k(t,y)\big). $$
	Then the IFS $\mathcal{J}:=\{J\times \mathbb{C};W_1,W_2,\dots,W_{N-1}\}$ has a unique attractor, \cite[Theorem 1]{MF1}, which is the graph of a function $h$ which satisfying
	the self-referential equation:
	$$h(t)= \alpha_k h \big(P_k^{-1}(t) \big)+ q_k \big(P_k^{-1}(t)\big), ~t \in J_k, k \in T.$$
	The above function $h$ is known as the fractal interpolation function (FIF). 
	\subsection{$\alpha$-Fractal Functions}
	To obtain a class of fractal functions with respect to a given continuous function on a compact interval in $\mathbb{R}$, we can adapt the idea of construction of FIF. The space of all complex-valued continuous functions defined on $J=[x_1,x_N]$ in $\mathbb{R}$ is denoted by $\mathcal{C}(J)$, with the sup norm. Let $f $ be a given function in $ \mathcal{C}(J)$, known as the germ function. For constructing the IFS, we consider the following assumptions
	\begin{enumerate}
		\item Let $\Delta:=\{x_1,x_2,\dots,x_N:x_1<x_2<\dots<x_N\}$ be a partition of $J=[x_1,x_N]$.
		
		\item  Let $\alpha_k: J \rightarrow \mathbb{C}$ be continuous functions with $\| \alpha_k\|_{\infty}=\max\{|\alpha_k(t)|:t\in J\} < 1$, for all $ k \in T $. These $\alpha_k$ are called the scaling functions and  $\alpha= \big(\alpha_1, \alpha_2, \dots, \alpha_{N-1}\big) \in \big(\mathcal{C}(J) \big)^{N-1}$ is called the  scaling vector.
		
		\item Let  $b:J \rightarrow \mathbb{C}$ be a continuous function such that $b \ne f$ and $b(x_1)=f(x_1), b(x_N)= f(x_N)$, named as the base function.
	\end{enumerate}
	Motivated by \cite{MF1,MF2}, Navascu\'{e}s \cite{N2} considered the following set of functons
	\begin{equation} \label{BE1}
		\begin{split}
			P_k(t) =&~ a_k t + d_k, \\
			\Psi_k(t,y)=&~ \alpha_k(t) y + f \big( P_k(t)\big) -\alpha_k(t) b(t).
		\end{split}
	\end{equation}
	Then the corresponding IFS $\mathcal{J}:=\{J\times \mathbb{C};W_1,W_2,\dots,W_{N-1}\}$, where
	$$W_k(t,y) = \Big(P_k(t), \Psi_k(t,y)\Big),$$
	has a unique attractor, which is the graph of a continuous function $f_{\Delta,b}^{\alpha}: J \rightarrow \mathbb{C}$ such that  $f_{\Delta,b}^{\alpha}(x_n)= f(x_n), n=1,2,\dots,N $. For simplicity, we denote $f_{\Delta,b}^{\alpha}$ by $f^\alpha$. The real valued $f^\alpha$ is widely known as $\alpha $-fractal function, see, for instance, \cite{VV,JV1,TV,SS2,JCN}.  Moreover, $f^{\alpha}$ satisfies the following equation
	\begin{equation}\label{FractalEquation}
		f^{\alpha}(t)= f(t)+\alpha_k(P_k^{-1}(t)).(f^{\alpha}- b)\big(P_k^{-1}(t)\big)~~~~\forall~~ t \in J,~~ k \in T.
	\end{equation}
	We can also treat  $f^{\alpha}$ as a ``fractal perturbation" of $f$.
	

	\section{Main Theorems}\label{se3}
	In the following lemma, we provide a relationship between Hausdorff dimension of complex-valued continuous function and Hausdorff dimension of its real and imaginary part. 
	\begin{lemma}\label{new222}
		Suppose $f: [a,b] \rightarrow \mathbb{C}$ is a continuous function and $g,h:[a,b] \to \mathbb{R}$ is real and imaginary part of $f$ respectively, that is,  $f= g +ih $. Then we have
		\begin{itemize}
			\item[(1)] \label{t1} $\dim_H (\text{G}(g+i h))\geq \max\{\dim_H (\text{G}(g)),\dim_H (\text{G}(h)) \}.$ 
			\item[(2)] $\dim_H (\text{G}(g+i h)) = \dim_H( \text{G}(g))$,  provided the imaginary part $ h$ is Lipschitz.
		\end{itemize}
		
	\end{lemma}
	\begin{proof}
		\begin{itemize}
			\item[(1)]  Let us define a mapping $\Phi : G(f) \to G(g)$ as follows
			$$\Phi(x,g(x)+i h(x) )=(x,g(x)).$$
			We aim to show that $\phi$ is a Lipschitz mapping. Using simple properties of norm, it follows that
			\begin{align*}
				\|\Phi&(x_1,g(x_1)+i h(x_1) )-\Phi(x_2,g(x_2)+i h(x_2) )\|^2\\
				=~&\|(x_1,g(x_1))-(x_2,g(x_2))\|^2\\
				=~&|x_1-x_2|^2 +|g(x_1)-g(x_2)|^2\\
				\leq~&|x_1-x_2|^2 +|g(x_1)-g(x_2)|^2
				+|h(x_1)-h(x_2)|^2\\
				=~&\|(x_1,g(x_1)+i h(x_1) )-(x_2,g(x_2)+i h(x_2)\|^2.
			\end{align*}
			That is,  $\Phi$ is a Lipschitz map. Now, Lipschitz invariance property of Hausdorff dimension yields $$\dim_H (\text{G}(f))\geq \dim_H (\text{G}(g)).$$
			On similar lines, we obtain  $$\dim_H (\text{G}(f))\geq \dim_H (\text{G}(h)).$$
			Comibining both of the above inequalities, we get 
			$$\dim_H (\text{G}(f))\geq \max\{\dim_H (\text{G}(g)),\dim_H (\text{G}(h)) \},$$
			completing the proof of item (i).
			\item[(2)]  In this part, we continue our proof with the same mapping $\Phi : G(f) \to G(g),$ defined by 
			$$\Phi(x,g(x)+i h(x) )=(x,g(x)).$$ 
			Here our aim is to show that $\Phi$ is a bi-Lipschitz map.
			From Part (1) of Lemma \ref{t1}, it is obvious that $\Phi$ is a Lipschitz map.
			And 
			\begin{align*}
				\|(x_1&,g(x_1)+i h(x_1) )-(x_2,g(x_2)+i h(x_2)\|^2\\
				=~& |x_1-x_2|^2 +|g(x_1)-g(x_2)|^2
				+|h(x_1)-h(x_2)|^2\\ 
				\leq~&|x_1-x_2|^2 +C_1^2 |x_1-x_2|^2
				+|g(x_1)-g(x_2)|^2\\ 
				\leq~& (1+C_1^2)\{\ |x_1-x_2|^2+|g(x_1)-g(x_2)|^2 \}\\=~&(1+C_1^2)\|(x_1,g(x_1))-(x_2,g(x_2))\|^2\\
				=~& (1+C_1^2) \|\Phi(x_1,g(x_1)+i h(x_1) )-\Phi(x_2,g(x_2)+i h(x_2) )\|^2.
			\end{align*}
			Therefore, $\Phi$ is a bi-Lipschitz map. In the light of the bi-Lipschitz invariance property of Hausdorff dimension, we get
			$$\dim_H(\text{G}(f))= \dim_H (\text{G}(g)),$$
			this completes the proof.
			
		\end{itemize}
		
	\end{proof}
	
	Now, we will present some results similar to the above in terms of other dimensions.
	\begin{proposition}
		Suppose $f: [a,b] \rightarrow \mathbb{C}$  is a continuous function and $g,h:[a,b] \to \mathbb{R}$ is real and imaginary part of $f$ respectively, that is, $f= g +ih $. Then we have
		$$\dim_P (\text{G}(g+i h))\geq \max\{\dim_P (\text{G}(g)),\dim_P (\text{G}(h)) \},$$     $$\overline{\dim}_B (\text{G}(g+i h))\geq \max\{\overline{\dim}_B (\text{G}(g)),\overline{\dim}_B (\text{G}(h)) \},$$                         $$\underline{\dim}_B (\text{G}(g+i h))\geq \max\{\underline{\dim}_B (\text{G}(g)),\underline{\dim}_B (\text{G}(h)) \}.$$
	\end{proposition}
	\begin{proof}
		The proof is similar to part (1) of Lemma \ref{t1}, hence we omit.
	\end{proof}
	
	\begin{proposition} Suppose $f: [a,b] \rightarrow \mathbb{C}$  is a continuous function and $g,h:[a,b] \to \mathbb{R}$ is real and imaginary part of $f$ respectively, that is $f= g +ih $. If $h$ is a Lipschitz function, then we have
		$$\dim_P(\text{G}(f))= \dim_P (\text{G}(g)),\overline{\dim}_B(\text{G}(f))= \overline{\dim}_B (\text{G}(g)),$$ and 
		$$\underline{\dim}_B(\text{G}(f))= \underline{\dim}_B (\text{G}(g)).$$
	\end{proposition}
	
	\begin{proof}
		The proof is similar to part (2) of Lemma \ref{t1}, hence omitted.  
	\end{proof}
	\begin{lemma}
		Suppose $f: [a,b] \rightarrow \mathbb{C}$  is a continuous function and $g,h:[a,b] \to \mathbb{R}$ is real and imaginary part of $f$ respectively, that is, $f= g +ih .$ If $h$ is a Lipschitz function on $[0,1]$, then
		$$\dim_H (\text{G}(g+i h)) =\dim_H (\text{G}(g+ h))=\dim_H (\text{G}(g, h))= \dim_H( \text{G}(g)),$$
		$$\overline{\dim}_B (\text{G}(g+i h)) =\overline{\dim}_B (\text{G}(g+ h))=\overline{\dim}_B (\text{G}(g, h))= \overline{\dim}_B( \text{G}(g)),$$
		$$\dim_P (\text{G}(g+i h)) =\dim_P (\text{G}(g+ h))=\dim_P (\text{G}(g, h))= \dim_P( \text{G}(g)).$$
	\end{lemma}
	\begin{proof}
		The mapping $\Phi :G(g+h)\to G(g) $ defined by 
		$$\Phi(x,(g(x)+h(x)))=(x,g(x))$$
		is a bi-Lipschitz map, see part(1) of Lemma \ref{t1}. Now bi-Lipschitz invariance property of Hausdorff dimension, gives
		\begin{equation}\label{e1}
			\dim_H (\text{G}(g+ h))=\dim_H( \text{G}(g)).
		\end{equation}
		Again we can show  $\Phi :G(g,h)\to G(g) $ defined by $$\Phi(x,(g(x),h(x)))=(x,g(x)),$$ is a bi-Lipschitz map, see part (2) of Lemma \ref{t1}. By using bi-Lipschitz invariance property of Hausdorff dimension, we get 
		\begin{equation}\label{e2}
			\dim_H (\text{G}(g, h))=\dim_H( \text{G}(g)).
		\end{equation}
		Further, by Lemma \ref{t1}, Equation \ref{e1} and Equation \ref{e2}, we get
		$$\dim_H (\text{G}(g+i h)) =\dim_H (\text{G}(g+ h))=\dim_H (\text{G}(g, h))= \dim_H( \text{G}(g)).$$
		Since upper box dimension, lower box dimension and packing dimension satisfy bi-Lipschitz invariance property, the rest follows.
	\end{proof} 
	\begin{lemma}\label{l2}
		Suppose $g,h :[a,b]\to \mathbb{R}$  are continuous functions. Then $g+ih :[a,b]\to \mathbb{C}$, $(g,h):[a,b] \to \mathbb{R}^2$ are continuous functions and $\dim_HG({g+ih})= \dim_HG{(g,h)}$.  
	\end{lemma}
	\begin{proof}
		Let us define a mapping $\Phi:G({g+ih})\to G{(g,h)}$ as follows
		$$\Phi(x,g(x)+ih(x))=(x,(g(x),h(x))).$$ We target to show that $\Phi$ is a bi-Lipschitz map. Performing simple calculations, we have
		\begin{align*}
			\|\Phi&(x_1,g(x_1)+ih(x_1)),\Phi(x_2,g(x_2)+ih(x_2))\|^2\\=
			&\|(x_1,(g(x_1),h(x_1))),(x_2,(g(x_2),h(x_2)))\|^2\\=&|x_1-x_2|^2 +|g(x_1)-g(x_2)|^2
			+|h(x_1)-h(x_2)|^2\\=&\|(x_1,g(x_1)+ih(x_1)),(x_2,g(x_2)+ih(x_2))\|^2.
		\end{align*}
		Therefore, $\Phi$ is a bi-Lipschitz map.
		By using bi-Lipschitz invariance property of Hausdorff dimension, we get $$\dim_HG({g+ih})= \dim_HG{(g,h)}.$$
		Since upper box dimension, lower box dimension and packing dimension also fulfill the bi-Lipschitz invariance property, we complete the proof.

	\end{proof}

	\begin{remark}
		The Peano space filling  curve $\textbf{g}:[0,1] \to [0,1] \times[0,1]$ is $\frac{1}{2}$-H\"older continuous, see details \cite{Kono}. The component functions satisfy $\dim_HG({g_1}) =\dim_HG({g_2})= 1.5 $. On the other hand, we have $\dim_HG(\textbf{g}\big) \geq 2.$ Now, consider a complex-valued mapping $f(x)=g_1(x)+ i g_2(x),$ and using Lemma \ref{l2}, $\dim_HG(f) \geq 2.$ From this example, we say that the upper bound of the Hausdorff dimension for the graph of a complex-valued function cannot be expressed in terms of its H\"{o}lder exponent as we do for real-valued function.  
	\end{remark}
	\par
	From above, it is clear that, dimensional results for complex-valued function and real-valued function are different. Now, we are ready to give some dimensional results for complex-valued fractal interpolation function.

	\par
	Define a metric $D$ on $ J\times \mathbb{C}$ by
	$$D((t_1,z_1),(t_2,z_2))= |t_1-t_2|+|z_1-z_2|~~~~~~~\forall~~(t_1,z_1),(t_2,z_2)\in J\times \mathbb{C}.$$
	Then $\big( J\times \mathbb{C}, D \big)$ is a complete metric space.
	\begin{theorem}
		Let $\mathcal{I}:=\{J\times \mathbb{C};W_1,W_2,\dots,W_{N-1}\}$ be the IFS defined in the construction of $f^\alpha$ such that $$  c_k D((t_1,z_1),(t_2,z_2) ) \le D(W_k(t_1,z_1) , W_k(t_2,z_2)) \le C_k D((t_1,z_1),(t_2,z_2)) ,$$ where  $(t_1,z_1),(t_2,z_2) \in J\times \mathbb{C}$ and $0 < c_k \le C_k < 1 ~ \forall~ k \in T .$ Then $r \le  \dim_H(G(f^{\alpha})) \le R  ,$ where $r$ and $R$ are given by  $ \sum\limits_{k=1}^{N-1} c_k^{r} =1$ and $ \sum\limits_{k=1}^{N-1} C_k^{R} =1$ respectively.
	\end{theorem}
	
	\begin{proof}
		For upper bound of $\dim_H(G(f^{\alpha}))$, Follow Proposition $9.6$ in \cite{Fal}. For the lower bound of $\dim_H(G(f^{\alpha}))$ we proceed as follows.
		
		Let $V = (x_1,x_N) \times \mathbb{C}.$ Then   $$W_i(V) \cap W_{j}(V)=\emptyset,$$ for each  $i\ne j \in T.$ Because  $$P_i\big((x_1,x_N)\big) \cap P_j\big((x_1,x_N)\big)=\emptyset,   ~~~~\forall~~i\ne j \in T.$$ We can observe that  $ V \cap G(f^{\alpha}) \ne \emptyset,$ this implies that  IFS $\mathcal{I}$ satisfies the SOSC. Then there exists an index $i\in T^*$ such that $W_i(G(f^{\alpha}))\subset V,$ where $T^*:=\cup_{n \in \mathbb{N}}\{1,2,\dots, N-1\}^n$. We denote  $ W_i(G(f^{\alpha}))$ by $ (G(f^{\alpha}))_i$ for any $i\in T^*$. Now, it is obvious that for each $n\in \mathbb{N}$, the sets $\{(G(f^{\alpha}))_{ji}: j \in T^n \}$ is  disjoint. Then, for each $n\in \mathbb{N}$, IFS $\mathcal{L}_n=\{W_{ji}: j \in T^n\}$ satisfies the hypothesis of Proposition $9.7$ in \cite{Fal}. Hence, by Proposition $9.7$ in \cite{Fal} if $A_n^* $ is an attractor of the IFS $\mathcal{L}_n,$ then $ r_n \le \dim_H(A_n^*)$, where $r_n$ is given by  $ \sum_{ j \in T^n} c_{ji}^{r_n} =1.$ Then  $  r_n \le \dim_H(A_n^*) \le \dim_H(G(f^{\alpha}))$ because $A_n^*\subset G(f^{\alpha})$. Suppose that $ \dim_H(G(f^{\alpha})) < r.$ This implies that  $ r_n < r $. Let $ c_{max}=\max\{c_1, c_2, \dots,c_{N-1}\}.$ We have
		$$
		c_{i}^{- r_n}  = \sum_{ j \in T^n} c_{j}^{r_n}\ \ge \sum_{ j \in T^n} c_{j}^{r} c_{j}^{\dim_H(G(f^{\alpha})) -r} \ge \sum_{ j \in T^n} c_{j}^{r} c_{max}^{n(\dim_H(G(f^{\alpha})) - r)}. $$
		This implies that $$c_{i}^{- r} \geq c_{max}^{n(\dim_H(G(f^{\alpha})) -r)}. $$ 
		We have a contradiction for large value of $n\in \mathbb{N} $. Therefore, we get $ \dim_H(G(f^{\alpha})) \ge  r,$ proving the assertion.
	\end{proof} 
	\begin{remark}
		In \cite{MR}, Roychowdhury estimated the Hausdorff and box dimension of attractor of the hyperbolic recurrent iterated function system consisting of bi-Lipschitz mappings under the open set condition using Bowen's pressure function and volume argument. Note that recurrent iterated function is a generalization of the iterated function, hence so is Roychowdhury's result. We should emphasize on the fact that in the above we provide a proof without using pressure function and volume argument. Our proof can be generalized to general complete metric spaces. 
	\end{remark}
	\begin{remark}
		This theorem can be compared with Theorem(2.4) in \cite{JV1}. 
	\end{remark}

	The H\"{o}lder space is defined as follows:$$ \mathcal{H}^{\sigma}(J ) := \{h:J \rightarrow \mathbb{C}: ~\text{h is H\"{o}lder continuous with exponent}~ \sigma \} .$$
	Note that $(\mathcal{H}^{\sigma}(J),\|.\|_\mathcal{H})$ is a Banach space, where $ \|h\|_{\mathcal{H}}:= \|h\|_{\infty} +[h]_{\sigma}$ and $$[h]_{\sigma} = \sup_{t_1\ne t_2} \frac{|h(t_1)-h(t_2)|}{|t_1-t_2|^{\sigma}}.$$
	
	
	\begin{theorem}\label{BBVL3}
		Let $f, b, \alpha \in  \mathcal{H}^{\sigma}(J )$ such that $b(x_1)=f(x_1)$ and $b(x_N)=f(x_N).$ Set $c:= \min\{a_k: k \in T \}$. If $  \frac{\|\alpha\|_{\mathcal{H}}}{c^\sigma}< 1 $, then  
		$f^{\alpha}$ is H\"{o}lder continuous with exponent $\sigma$.
	\end{theorem}
	\begin{proof}
		Let us define $ \mathcal{H}^{\sigma}_f(J ):= \{ h \in \mathcal{H}^{\sigma}(J ): h(x_1)=f(x_1), ~h(x_N)=f(x_N) \}.$ 
		By basic real analysis technique, we may see that $\mathcal{H}^{\sigma}_f(J )$ is a closed subset of $\mathcal{H}^{\sigma}(J).$ Since $(\mathcal{H}^{\sigma}(J),\|.\|_\mathcal{H})$ is a Banach space, it implies that $\mathcal{H}^{\sigma}_f(J )$  will be a complete metric space with respect to metric induced by $\|.\|_\mathcal{H}.$ We define a map $S: \mathcal{H}^{\sigma}_f(J) \rightarrow \mathcal{H}^{\sigma}_f(J )$ by $$ (Sh)(t)=f(t)+\alpha_k(P_k^{-1}(t)) ~(h-b)(P_k^{-1}(t))  $$
		$\forall~~ t \in J_k $ where $k \in T.$ We shall show that $S$ is well-defined and contraction map on $\mathcal{H}^{\sigma}_f(J)$.
		\begin{equation*}
			\begin{split}
				[Sh]_\sigma   = &\max_{k  \in T} \sup_{t_1 \ne t_2, t_1,t_2  \in J_k} \frac{|Sh(t_1)-Sh(t_2)|}{|t_1-t_2|^{\sigma}}\\
				\le&  \max_{k  \in T} \Bigg[ \sup_{t_1 \ne t_2, t_1,t_2 \in J_k} \frac{|f(t_1)-f(t_2)|}{|t_1-t_2|^{\sigma}}\\
				& +  \sup_{t_1\ne t_2, t_1,t_2 \in J_k} \frac{|\alpha_k(P_k^{-1}(t_1))| \Big|(h-b)(P_k^{-1}(t_1))-(h-b)(P_k^{-1}(t_2))\Big|}{|t_1-t_2|^{\sigma}}\\ & + \sup_{t_1\ne t_2, t_1,t_2 \in J_k} \frac{|(h-b)(P_k^{-1}(t_2))| \Big|\alpha_k(P_k^{-1}(t_1))-\alpha_k(P_k^{-1}(t_2))\Big|}{|t_1-t_2|^{\sigma}}\Bigg]\\
				\le & ~[f]_{\sigma}+ \frac{\|\alpha\|_{\infty}}{c^{\sigma}} \big( [h]_{\sigma}+[b]_{\sigma} \big)+ \frac{\|h-b\|_{\infty}}{c^{\sigma}} [\alpha]_{\sigma},
			\end{split}
		\end{equation*}
		where $[\alpha]_{\sigma}= \max\limits_{k \in T} \sup\limits_{t_1 \ne t_2, t_1,t_2 \in J} \frac{ |\alpha_k(t_1)-\alpha_k(t_2)|}{|t_1-t_2|^{\sigma}}.$ Let $g, h \in \mathcal{H}^{\sigma}_f(J )$, we have
		\begin{equation*}
			\begin{aligned}
				\|Sg -Sh\|_{\mathcal{H}} &= \|Sg -Sh\|_{\infty} + [Sg-Sh]_{\sigma}\\
				&\le \| \alpha \|_{\infty} \|g -h\|_{\infty} + \frac{\|\alpha\|_{\infty}}{c^{\sigma}} [g-h]_{\sigma}+\frac{\|g-h\|_{\infty}}{c^{\sigma}} [\alpha]_{\sigma}\\
				&\le \frac{\|\alpha\|_{\mathcal{H}}}{c^{\sigma}} \| g-h\|_{\mathcal{H}}.
			\end{aligned}
		\end{equation*}
		This implies that $S$ is well-defined map on $\mathcal{H}^{\sigma}_f(J )$.    
		Since $\frac{\|\alpha\|_{\mathcal{H}}}{c^{\sigma}} < 1$,  $S$ is a contraction map. By the application of Banach contraction mapping theorem, $S$ has a unique fixed point $f^{\alpha} \in \mathcal{H}^{\sigma}_f(J )$. Hence we are done.
	\end{proof}

	\begin{theorem}\label{mainthm}
		Let germ function $f$, base function $b$ and scaling function $ \alpha_j$ be complex-valued functions such that
		\begin{equation}\label{Hypo}
			\begin{aligned}
				& |f(t_1) -f(t_2)| \le l_f |t_1-t_2 |^{\sigma},\\&
				|b(t_1) -b(t_2)| \le l_b |t_1-t_2|^{\sigma},\\&
				|\alpha_j(t_1) -\alpha_j(t_2)| \le l_{\alpha} |t_1-t_2|^{\sigma}
			\end{aligned}
		\end{equation}
		for each $t_1,t_2 \in J ,j \in T,$ and for some $l_f, l_b, l_{\alpha} > 0, \sigma\in (0,1]$. Let $f_1,f_2$ be component of $f$, $b_1,b_2$ be component of $b$,$\alpha_j^{1},\alpha_j^{2}$ be component of $\alpha_j$ and $f^{\alpha}_1,f^{\alpha}_2$ be component of $f^{\alpha}$.   
		Also, consider constants $l_{f_i}, \delta_0> 0$ such that for all $t_1 \in J $ and $\delta < \delta_0$, there exists $t_2\in J $ with $|t_1-t_2| \le \delta$ and $$ |f_i(t_1)-f_i(t_2)| \ge l_{f_i} |t_1 -t_2|^{\sigma}~~\text{ for}~ i\in\{1,2\} .$$
		If ~~~$ \|\alpha\|_{\mathcal{H}}< c^\sigma ~\min\Big\{1,\frac{l_{f_1}-2(\|b\|_{\infty}+M)l_{\alpha}c^{-\sigma} }{2(k_{f,b,\alpha}+l_b)},\frac{l_{f_2}-2(\|b\|_{\infty}+M)l_{\alpha}c^{-\sigma} }{2(k_{f,b,\alpha}+l_b)}\Big\},$  then we have $$1 \le \dim_H\big(G{(f^{\alpha}_i)}\big) \le \dim_B\big(G{(f^{\alpha}_i)}\big) = 2 - \sigma~~~\text{for}~~i=1,2.$$Moreover, $1 \le \dim_H\big(G{(f^{\alpha})}\big) \le \dim_B\big(G{(f^{\alpha})}\big) \ge  2 - \sigma.$
	\end{theorem}
	\begin{proof}
		Since $ \|\alpha\|_{\mathcal{H}}< c^\sigma$, Theorem (3.11) yields that the fractal version $f^{\alpha}$ of $f$ is H\"older continuous with the exponent $\sigma$. Consequently, $$|f^{\alpha}_i(t_1)-f^{\alpha}_i(t_2)|\leq |f^{\alpha}(t_1)-f^{\alpha}(t_2)| \le k_{f,b,\alpha} |t_1-t_2|^{\sigma}$$ for some $ k_{f,b,\alpha}> 0$ and for $i=1,2.$ 
		Firstly, we try to give an upper bound for upper box dimension of  $G(f^{\alpha}_i) $ for $i=1,2$ as follows:
		For $ \delta \in (0,1) $, let $m$ be the smallest natural number greater than or equal to $\frac{1}{\delta}$  and $N_{\delta}(G({f^{\alpha}_i}))$ be the number of $\delta$-mesh that can intersect with $G(f^{\alpha}_i)$, we have
		\begin{equation}
			\begin{aligned}
				N_{\delta}(G({f^{\alpha}_i})) & \le 2m+ \sum_{r=0}^{m-1}  \Bigg(  { \frac{R_{f^{\alpha}_i}[(r\delta,(r+1)\delta]}{\delta} } \Bigg)\\& \leq 2 \bigg({\frac{1}{\delta}}+1\bigg) + \sum_{r=0}^{m-1}   \frac{R_{f^{\alpha}_i}[r\delta,(r+1)\delta]}{\delta}\\& \leq  2 \bigg({\frac{1}{\delta}}+1\bigg)+ \sum_{r=0}^{m-1}   k_{f,b,\alpha} \delta^{\sigma -1} .
			\end{aligned}
		\end{equation}
		From this, we conclude that
		$$\overline{\dim}_B\big(G(f^{\alpha}_i)\big) =\varlimsup_{\delta \rightarrow 0} \frac{\log N_{\delta}(G(f^{\alpha}_i))}{- \log \delta}\le 2- \sigma~~~~~\forall~~~ i=1,2.$$
		Next, we will prove that  $\underline{\dim}_B\big(G(f^{\alpha}_i)\big) \ge 2 - \sigma~~~~~~~~~\forall~~ i=1,2$. For this, using the self-referential equation, we can write
		\begin{equation}\label{new333}
			\begin{aligned}
				f^{\alpha}_1(t)=&f_1(t)+\alpha_k^1\big(P_k^{-1}(t)\big)  \big[f^{\alpha}_1\big(P_k^{-1}(t)\big) - b_1\big(P_k^{-1}(t)\big)\big]\\&-\alpha_k^2\big(P_k^{-1}(t)\big)  \big[f^{\alpha}_2\big(P_k^{-1}(t)\big)- b_2\big(P_k^{-1}(t)\big)\big]
			\end{aligned}
		\end{equation}
		
		for every $t \in J_k $ and $k \in T.$
		Let  $ t_1 ,t_2 \in J_k $ such that $|t_1-t_2| \le \delta.$ From Equation \ref{new333}, we have
		\begin{align*}
			|f^{\alpha}_1(t_1)- f^{\alpha}_1(t_2)| =& \Big| f_1(t_1)-f_1(t_2) \\&+ \alpha_k^1\big(P_k^{-1}(t_1)\big) ~ f^{\alpha}_1\big(P_k^{-1}(t_1)\big) - \alpha_k^1\big(P_k^{-1}(t_2)\big) ~ f^{\alpha}_1\big(P_k^{-1}(t_2)\big)\\& - \alpha_k^1\big(P_k^{-1}(t_1\big) ~ b_1\big(P_k^{-1}(t_1)\big) + \alpha_k^1\big(P_k^{-1}(t_2)\big) ~ b_1\big(P_k^{-1}(t_2)\big) \\&- \alpha_k^2\big(P_k^{-1}(t_1)\big) ~ f^{\alpha}_2\big(P_k^{-1}(t_1)\big) + \alpha_k^2\big(P_k^{-1}(t_2)\big) ~ f^{\alpha}_2\big(P_k^{-1}(t_2)\big)\\& + \alpha_k^2\big(P_k^{-1}(t_1)\big) ~ b_2\big(P_k^{-1}(t_1)\big) - \alpha_k^2\big(P_k^{-1}(t_2)\big) ~ b_2\big(P_k^{-1}(t_2)\big) \Big|\\
			\ge & | f_1(t_1)-f_1(t_2)| - \|\alpha\|_{\infty} ~ \Big|f^{\alpha}_1\big(P_k^{-1}(t_1)\big) -  f^{\alpha}_1\big(P_k^{-1}(t_2)\big) \Big|\\& - \|\alpha\|_{\infty} ~ \Big|b_1\big(P_k^{-1}(t_1)\big) -   b_1\big(P_k^{-1}(t_2)\big) \Big| - \big(\|b\|_{\infty}+\|f^{\alpha}\|_{\infty}\big)\\& ~\Big|\alpha_k^1\big(P_k^{-1}(t_1)\big)-\alpha_k^1\big(P_k^{-1}(t_2)\big)
			\Big| -\|\alpha\|_{\infty} ~ \Big|f^{\alpha}_2\big(P_k^{-1}(t_1)\big) \\&-  f^{\alpha}_2\big(P_k^{-1}(t_2)\big) \Big| - \|\alpha\|_{\infty} ~ \Big|b_2\big(P_k^{-1}(t_1)\big) -   b_2\big(P_k^{-1}(t_2)\big) \Big|\\& - \big(\|b\|_{\infty}+\|f^{\alpha}\|_{\infty}\big) ~\Big|\alpha_k^2\big(P_k^{-1}(t_1)\big)-\alpha_k^2\big(P_k^{-1}(t_2)\big)
			\Big|.
		\end{align*}
		With the help of Equation (\ref{Hypo}), we get
		\begin{equation*}
			\begin{aligned}
				|f^{\alpha}_1(t_1)- f^{\alpha}_1(t_2)|
				\ge & ~l_{f_1} | t_1-t_2|^{\sigma}-2\|\alpha\|_{\infty} ~ k_{f,b,\alpha}\Big|P_k^{-1}(t_1) -  P_k^{-1}(t_2) \Big|^{\sigma}\\& -2 \|\alpha\|_{\infty}~ l_b \Big|P_k^{-1}(t_1) -   P_k^{-1}(t_2) \Big|^{\sigma}\\&-2 \big(\|b\|_{\infty}+M\big) l_{\alpha} ~\Big|P_k^{-1}(t_1) -   P_k^{-1}(t_2) \Big|^{\sigma}\\
				\ge & ~l_{f_1} | t_1-t_2|^{\sigma}-2\|\alpha\|_{\infty} ~ k_{f,b,\alpha} a^{-\sigma} | t_1-t_2|^{\sigma}\\& - 2\|\alpha\|_{\infty} ~ l_b c^{-\sigma} | t_1-t_2|^{\sigma}\\&- 2\big(\|b\|_{\infty}+M\big)c^{-\sigma} l_{\alpha} ~|t_1-t_2|^{\sigma}\\
				= & ~\Big(l_{f_1}-2(k_{f,b,\alpha}+ l_b)\|\alpha\|_{\infty} c^{-\sigma}-2\big(\|b\|_{\infty}+M\big)c^{-\sigma} l_{\alpha}\Big) | t_1-t_2|^{\sigma}.
			\end{aligned}
		\end{equation*}
		Set  $L:=l_{f_1}-2(k_{f,b,\alpha}+ l_b)\|\alpha\|_{\infty} c^{-\sigma}-2\big(\|b\|_{\infty}+M\big)c^{-\sigma} l_{\alpha}.$ Then by the given condition $L > 0 $. Let $\delta= c^{n}$ for $n\in \mathbb{N}$ and $w$ be the smallest natural number greater than or equal to $\frac{1}{c^m}$. We estimate 
		\begin{equation*}
			\begin{aligned}
				N_{\delta}(G(f^{\alpha}_1)) &\ge  \sum_{r=0}^{w} \max\Big\{1, \big(  c^{-n} R_{f^{\alpha}_1}[r\delta,(r+1)\delta]\big ) \Big\}\\
				& \ge  \sum_{r=0}^{w}  { \big( c^{-n} R_{f^{\alpha}_1}[r\delta,(r+1)\delta] \big)} \\ & \ge  \sum_{r=0}^{w }  {L c^{-n}   c^{n \sigma }}\\
				& =   L c^{n( \sigma-2)}   .
			\end{aligned}
		\end{equation*}
		By using the above inequality for $N_{\delta}(G(f^{\alpha}_1))$, we obtain
		\begin{equation*}
			\begin{aligned}
				\underline{\dim}_B\big(G(f^{\alpha}_1)\big) =\varliminf_{\delta \rightarrow 0}\frac{ \log\Big( N_{\delta}(G(f^{\alpha}_1))\Big)}{- \log (\delta)}
				& \ge \varliminf_{n \rightarrow \infty}\frac{ \log\Big(L  c^{n(\sigma - 2)}   \Big) }{-n \log c}\\ & =
				2- \sigma,
			\end{aligned}
		\end{equation*}
		Similarly, we get $$\underline{\dim}_B\big(G(f^{\alpha}_2)\big) \ge 2 - \sigma,$$
		establishing the result.
	\end{proof}

	\begin{definition}
		A complex-valued function $ g:J \rightarrow \mathbb{C}$ is said to be of bounded variation if the total variation $V(g,J) $ of $g$ defined by $$V(g,J)= \sup_{Q=(y_0,y_1, \dots,y_m) ~~\text{partition of } J}~ \sum_{k=1}^{m} |g(y_i)-g(y_{i-1})|,$$ is finite.
		The space of all bounded variation functions on $J,$ denoted by $\mathcal{BV}(J,\mathbb{C}),$ forms a Banach space with respect to the norm
		$\|g\|_{\mathcal{BV}}:= |g(y_0)|+ V(g,J).$
		
	\end{definition}

	\begin{theorem}
		If $f, b\in \mathcal{C}(J,\mathbb{C})\cap \mathcal{BV}(J,\mathbb{C})$ and $\alpha_k\in \mathcal{C}(J,\mathbb{C})\cap \mathcal{BV}(J,\mathbb{C})~~~\forall~~k\in T$ such that $\|\alpha\|_{\mathcal{BV}}< \frac{1}{2(N-1)}.$ Then $f^\alpha \in  \mathcal{C}(J,\mathbb{C})\cap \mathcal{BV}(J,\mathbb{C})$. Moreover, $\dim_H(G(f^{\alpha}))=\dim_B(G(f^{\alpha}))=1 .$
	\end{theorem}
	\begin{proof}
		Following Theorem (3.11) and \cite[Theorem 3.24 ]{JV1}, one may complete the proof.
	\end{proof}
	\begin{remark}
		The above theorem will reduce to \cite[Theorem 3.24]{JV1} when all functions $f,b $ and $\alpha_k$ are real-valued. 
	\end{remark}

	\bibliographystyle{amsplain}

\begin{thebibliography}{10}
		\bibitem{SP} A. Sahu, A. Priyadarshi, On the box-counting dimension of graphs of harmonic functions on the Sierpi\'{n}ski gasket, J. Math. Anal. Appl. 487 (2020) 124036.
		\bibitem{Schief} A. Schief, Separation properties for self-similar sets, Proc. Amer. Math. Soc. 122 (1994) 111-115.
		\bibitem{M&H} D. P. Hardin, P. R. Massopust, The capacity for a class of fractal functions, Comm. Math. Phys. 105 (1986) 455-460.
		\bibitem{HM} D. P. Hardin, P. R. Massopust, Fractal interpolation functions from $\mathbb{R}^n$ to $\mathbb{R}^m$ and their projections, Zeitschrift f\"ur Analysis u. i. Anw. 12 (1993) 535-548.
		\bibitem{BFY} F. Bayart, Y. Heurteaux, On the Hausdorff dimension of graphs of prevalent continuous functions on compact sets, Further developments in fractals and related fields, (2013) 25-34.
		\bibitem{WY} H.Y. Wang, J.S. Yu, Fractal interpolation functions with variable parameters and their analytical properties, J. Approx. Theory 175 (2013) 1-18.
		\bibitem{LW}  J. Liu, J. Wu, A remark on decomposition of continuous functions, Journal of Mathematical Analysis and Applications, 401 (2013), 404-406.
		\bibitem {Fal} K. J. Falconer, Fractal Geometry: Mathematical Foundations and Applications, John Wiley Sons Inc., New York, 1999.
		\bibitem{F&F} K. J. Falconer, J. M. Fraser, The horizon problem for prevalent surfaces, Math. Proc. Cambridge Philos. Soc. (2) 151 (2011) 355–372.
		\bibitem {N2} M. A. Navascu\'{e}s, Fractal polynomial interpolation, Z. Anal. Anwend. 25(2) (2005) 401-418.
		\bibitem {MF1} M. F. Barnsley, Fractal functions and interpolation, Constr. Approx. 2 (1986) 303-329.
		\bibitem {MF2}  M. F. Barnsley, Fractal Everywhere, Academic Press, Orlando, Florida, 1988.
		\bibitem{MF6} M. F. Barnsley, J. Elton, D. P. Hardin, P. R. Massopust,
		Hidden variable fractal interpolation functions, SIAM J. Math.
		Anal., 20(5) (1989) 1218-1248.
		\bibitem{B&H} M. F. Barnsley, A. N. Harrington, The calculus of fractal interpolation functions, J. Approx. Theory, 57 (1989) 14-34.
		\bibitem{MF3} M. F. Barnsley, P. R. Massopust,
		Bilinear fractal interpolation and box dimension,
		J. Approx. Theory 192 (2015) 362-378.
		\bibitem{MR} M. K. Roychowdhury, Hausdorff and upper box dimension estimate of hyperbolic recurrent sets, Israel J. Mathematics 201 (2014) 507-523.
		\bibitem{Kono} N. Kono, On self-affine functions, Japan J. Appl. Math. 3 (1986) 259-269.
		\bibitem{MP} P. R. Massopust, Vector-valued fractal interpolation functions and their box dimension, Aequationes Math. 42 (1991) 1-22.
		\bibitem{pt} P. Wingren, Dimensions of graphs of functions and lacunary decompositions of spline approximations, Real Anal. Exchange (2000) 17-26.
		\bibitem{w&l} R. D. Mauldin, S. C. Williams, On the Hausddorff dimension of some graphs, Trans. Amer. Math. Soc. 298 (1986) 789-803.
		\bibitem{SS2} S. Chandra, S. Abbas, The calculus of bivariate fractal interpolation surfaces, Fractals 29(3) (2021) 2150066.
		\bibitem{JCN} S. Jha, A. K. B. Chand, and M. A. Navascu\'{e}s, Approximation by shape preserving fractal functions with variable scalings, Calcolo 58(1) (2021) 24.
		\bibitem{JV1} S. Jha, S. Verma, Dimensional analysis of $\alpha$-fractal functions, Results Math 76 186 (2021).
		\bibitem{VV} S. Verma, P. Viswanathan, A revisit to $\alpha$-fractal function and box dimension of its graph, Fractals, 27(6) (2019) 1950090.
		\bibitem{TV} V. Agrawal, T. Som, Fractal dimension of $\alpha$-fractal function on the Sierpiński Gasket, The European Physical Journal Special Topics, (2021) 1-7. 
		%
		%
		%
		%
		%
		
		
		
		%
		
		
		
		
		
		
		
		
		
		
		
		
		
		
		
		
		
		
		
		
		
		
		
		
		
		
		
		
		
		
		
		
		
		
	\end{thebibliography}

\end{document}